\documentclass[11pt]{article}

\usepackage[margin=3cm]{geometry}
\usepackage{amsmath,amssymb,amsthm}
\usepackage[hidelinks]{hyperref}

\newtheorem{theorem}{Theorem}
\newtheorem{lemma}[theorem]{Lemma}
\newtheorem{remark}{Remark}

\newcommand{\J}{J}
\newcommand{\Mav}{\widehat{M}_{\mu}}

\title{An Exact Obstruction to Uniform Average Mixing on $P_{11}$}
\author{Musung Kang\thanks{Department of Mathematical Sciences,
Seoul National University, Seoul, Republic of Korea.
Email: \texttt{musung098@snu.ac.kr}}}
\date{}

\begin{document}
\maketitle

\begin{abstract}
We prove that the path $P_{11}$ does not admit uniform average mixing
under any probability distribution on $\mathbb R$, answering a question
of Baptista, Coutinho, and Marques in the negative.  The proof is exact:
we construct an explicit rational symmetric matrix $Y$ such that
$\langle Y,M(t)\rangle_F=1$ for every $t\in\mathbb R$, whereas
$\langle Y,J/11\rangle_F=12/11$.
\end{abstract}

\section{Statement}

Let $A$ be the adjacency matrix of a graph on $n$ vertices and define
\begin{align*}
U(t)=\exp(itA),
\quad
M(t)=U(t)\circ\overline{U(t)}.
\end{align*}
Thus $M_{pq}(t)=|U_{pq}(t)|^2$.  For real matrices $X$ and $Z$ of the same
size, their Frobenius inner product is
\begin{align*}
\langle X,Z\rangle_F
:=\operatorname{tr}(X^{\mathsf T}Z)
=\sum_{p,q}X_{pq}Z_{pq}.
\end{align*}
For a probability measure $\mu$ on
$\mathbb R$, set
\begin{align*}
\Mav=\int_{\mathbb R}M(t)\,d\mu(t).
\end{align*}
Following \cite{BCM24}, uniform average mixing under $\mu$ means
$\Mav=\frac{1}{n}\J$, where $\J$ denotes the all-ones matrix. We use standard measure theory; see \cite[Chapter~3]{Folland}.

\begin{lemma}[Constant linear functional]\label{lem:linf}
Suppose that a real symmetric matrix $Y$ satisfies
\begin{align*}
g_Y(t):=\langle Y,M(t)\rangle_F=\gamma
\quad\text{for every }t\in\mathbb R.
\end{align*}
Then
\begin{align*}
\langle Y,\Mav\rangle_F=\gamma
\end{align*}
for every probability measure $\mu$.  Consequently, if
\begin{align*}
\left\langle Y,\frac{1}{n}\J\right\rangle_F\ne\gamma,
\end{align*}
then $\Mav\ne\frac{1}{n}\J$ for every $\mu$.

More generally, 
the same conclusion holds for every finite signed measure
$\nu$ of total mass one 
for which $g_Y$ is integrable with respect to the total-variation measure $|\nu|$. 
\end{lemma}

\begin{proof}
Let $\nu$ be such a signed measure.  Since the sum defining the Frobenius inner
product is finite, linearity of integration gives
\begin{align*}
\left\langle Y,\int_{\mathbb R}M(t)\,d\nu(t)\right\rangle_F
=\int_{\mathbb R}g_Y(t)\,d\nu(t)
=\gamma\nu(\mathbb R)
=\gamma.
\end{align*}
The assertion for a probability measure is the special case $\nu=\mu$.
Moreover, the entries of $M(t)$ lie in $[0,1]$, so $g_Y$ is bounded; hence the
stated integrability is automatic for every finite signed measure.  The last
assertion follows by applying the same linear functional to $\frac{1}{n}\J$.
\end{proof}

\begin{theorem}\label{thm:p11}
For the adjacency matrix of $P_{11}$, no probability measure $\mu$ on
$\mathbb{R}$ achieves uniform average mixing.  In other words, for every
probability measure $\mu$ on $\mathbb{R}$,
\begin{align*}
\Mav \neq \frac{1}{11}\J.
\end{align*}
Thus $P_{11}$ answers the question of \cite[\S 7]{BCM24} in the negative.
\end{theorem}

\section{The exact matrix}
Number the vertices of $P_{11}$ by $1, \ldots, 11$. 
Define the real symmetric matrix $Y$ by
\begin{align*}
Y_{pp}&=(-1)^{p+1}, \quad 1 \le p \le 5, \\
Y_{p,12-p}=Y_{12-p,p}&=-\frac{(-1)^{p+1}}{2}, \quad 1 \le p\le5, \\
Y_{p6}=Y_{6p}&=2\sin^2\frac{p\pi}{6}, \quad 1 \le p \le5,
\end{align*}
and set every other entry equal to zero. The eigenvalues of $A(P_{11})$
\[
\theta_k=2\cos\frac{k\pi}{12},
\quad 1\le k\le11
\]
are simple. Their rank-one spectral idempotents are
$E_k=v_kv_k^{\mathsf T}$, where
\[
(v_k)_p=\frac{1}{\sqrt6}\sin\frac{pk\pi}{12}.
\]
Consequently,
\[
(E_k)_{pq}
=\frac{1}{6}\sin\frac{pk\pi}{12}\sin\frac{qk\pi}{12},
\quad 1\le p,q,k\le11.
\] 
Moreover, by the standard spectral decomposition of a path
(see \cite[\S 1.4.4]{BH}),
\[
A=\sum_{k=1}^{11}\theta_kE_k,
\quad
U(t)=\sum_{k=1}^{11}e^{i\theta_kt}E_k.
\]
Therefore,
\begin{align*}
M(t)
=\sum_{k=1}^{11}E_k\circ E_k
+2\sum_{1\le k<l\le11}
\cos\bigl((\theta_k-\theta_l)t\bigr)(E_k\circ E_l).
\end{align*}
For $1\le k,l\le11$, write
\begin{align*}
C_{kl}:=\langle Y,E_k\circ E_l\rangle_F.
\end{align*}

\begin{lemma}\label{lem:coeff}
The diagonal coefficients are
\begin{align*}
C_{kk}=
\begin{cases}
    \frac{1}{6},    &   k \text{ odd}, \\
    0,  &   k \text{ even}.
\end{cases}
\end{align*}
For $k<l$, the only nonzero coefficients are
\begin{align*}
C_{kl}=
\begin{cases}
\frac{1}{12}, & 6 \in \{k,l\} \text{ and the other index is odd}, \\
-\frac{1}{12}, &(k,l)\in\mathcal{S},
\end{cases}
\end{align*}
where 
$\displaystyle \mathcal{S}=\{(1,3),(1,5),(3,7),(5,9),(7,11),(9,11)\}$.
\end{lemma}

\begin{proof}
Set
\begin{align*}
x_p=\sin{\frac{pk\pi}{12}}\sin{\frac{pl\pi}{12}}.
\end{align*}
Note that
\begin{align*}
x_{12-p}=(-1)^{k+l}x_p.
\end{align*}
Substitution of the entries of $Y$ then gives the exact identity
\begin{align*}
36C_{kl} = \bigl(1-(-1)^{k+l}\bigr)A_{kl} + 4x_6B_{kl},
\end{align*}
where
\begin{align*}
A_{kl}=\sum_{p=1}^{5}(-1)^{p+1}x_p^2, \quad
B_{kl}=\sum_{p=1}^{5}\sin^2\frac{p\pi}{6} x_p.
\end{align*} 
Consider all three parity cases as shown in the table below.
For an odd index $r$, put $\displaystyle \varepsilon_r=(-1)^{(r-1)/2}
=\sin\frac{r\pi}{2}.$
Since
$\displaystyle x_6=\sin\frac{k\pi}{2}\sin\frac{l\pi}{2}$,
the identity above reduces in the three parity cases as follows:
\begin{align*}
\begin{array}{c|c|c|c}
\text{parities of } k , l &1-(-1)^{k+l} &x_6 &36C_{kl}\\ \hline
\text{Case 1 : even, even} & 0 & 0 & 0\\
\text{Case 2 : one odd, one even} & 2 & 0 & 2 A_{kl} \\
\text{Case 3 : odd, odd} & 0 & \varepsilon_k\varepsilon_l & 4\varepsilon_k\varepsilon_lB_{kl}
\end{array}
\end{align*}
Evaluations below rely on a single orthogonality relation with
$\omega=e^{i\pi/6}$ a primitive twelfth root of unity,
\begin{align*}
\sum_{p=0}^{11}\omega^{pm}
=12\cdot\mathbf{1}[12\mid m],
\end{align*}
where $\mathbf{1}[\mathcal P]$ equals $1$ if $\mathcal P$ holds and $0$
otherwise.

\medskip
\noindent\emph{Case 1: both indices are even.}
Both terms in the formula for $36C_{kl}$ vanish.  Hence
\begin{align*}
C_{kl}=0
\end{align*}
whenever $k$ and $l$ are both even. 
This includes $C_{kk}=0$ for even $k$.

\medskip
\noindent\emph{Case 2: exactly one index is even.}
By symmetry in $k$ and $l$, we may assume without loss of generality
that $k=o$ is odd and $l=e$ is even.  In this case $x_6=0$, and hence
\begin{align*}
C_{oe}=\frac{1}{18}A_{oe}.
\end{align*}
The summand of $A_{oe}$ is invariant under $p\mapsto12-p$, and
$x_p$ vanishes at $p=0$ and $p=6$.  Therefore,
\begin{align*}
A_{oe}
=\frac{1}{2}\sum_{p=0}^{11}(-1)^{p+1}x_p^2.
\end{align*}
Using
\[
\sin^2\frac{pr\pi}{12}
=\frac{1}{4}\bigl(2-\omega^{pr}-\omega^{-pr}\bigr),
\]
we obtain
\[
x_p^2
=\frac{1}{16}\sum_{a,b\in\{0,\pm1\}}
c_{ab}\omega^{p(ao+be)},
\]
where $c_{00}=4$, $c_{\pm1,0}=c_{0,\pm1}=-2$, and
$c_{ab}=1$ when $a,b\ne0$.  Since
\[
(-1)^{p+1}=-\omega^{6p},
\]
it follows that
\[
\begin{aligned}
A_{oe}
&=-\frac{1}{32}\sum_{a,b}c_{ab}
  \sum_{p=0}^{11}\omega^{p(6+ao+be)}\\
&=-\frac{3}{8}\sum_{a,b}c_{ab}\,
  \mathbf1[ao+be\equiv6\pmod{12}],
\end{aligned}
\]
where the second equality follows from orthogonality.

Since $ao+be$ has the parity of $a$, the congruence forces $a=0$.
The choice $b=0$ is impossible, while for both $b=1$ and $b=-1$
the congruence $be\equiv6\pmod{12}$ holds precisely when $e=6$.
Hence
\begin{align*}
A_{oe}
=-\frac{3}{8}(-2-2)\mathbf1[e=6]
=\frac{3}{2}\mathbf1[e=6].
\end{align*}
Consequently,
\begin{align*}
C_{oe}=
\begin{cases}
\frac{1}{12}, & e=6,\\
0, & e=2,4,8,\text{ or }10.
\end{cases}
\end{align*}

\medskip
\noindent\emph{Case 3: both indices are odd.} 
Since $C_{kl}=C_{lk}$, we may assume $k\le l$.
In this case the
$A_{kl}$ term vanishes and $x_6=\varepsilon_k\varepsilon_l$, so
\begin{align*}
C_{kl}=\frac{\varepsilon_k\varepsilon_l}{9}B_{kl}.
\end{align*}
Put
\begin{align*}
u=\frac{k-l}{2},
\quad
v=\frac{k+l}{2},
\end{align*}
which are integers with $u+v=k$ odd, and define
\begin{align*}
d(m)=\sum_{p=1}^{5}\sin^2\frac{p\pi}{6}\cos\frac{pm\pi}{6},
\quad m\in\mathbb Z.
\end{align*}
Product-to-sum gives
\begin{align*}
B_{kl}=\frac{1}{2}\bigl(d(u)-d(v)\bigr).
\end{align*}
The summand of $d(m)$ is invariant under $p\mapsto12-p$, and the terms with
$p=0$ and $p=6$ vanish.  Therefore, using
$\displaystyle \sin^2\frac{p\pi}{6}=\frac{1}{4}\bigl(2-\omega^{2p}-\omega^{-2p}\bigr)$ and
$\displaystyle \cos\frac{pm\pi}{6}=\frac{1}{2}\bigl(\omega^{pm}+\omega^{-pm}\bigr)$,
\begin{align*}
2d(m)
=\frac{1}{8}\sum_{p=0}^{11}
\bigl(2-\omega^{2p}-\omega^{-2p}\bigr)
\bigl(\omega^{pm}+\omega^{-pm}\bigr),
\end{align*}
and the orthogonality relation gives
\begin{align*}
d(m)
=3\cdot\mathbf{1}[ 12\mid m ]
-\frac{3}{2}\cdot\mathbf{1}[ m\equiv\pm2\ (\mathrm{mod}\ 12) ].
\end{align*}
In particular, $d(m)=0$ whenever $m$ is odd.

\medskip
\noindent
Since $u+v$ is odd, exactly one of $u,v$ is even; let $w$ denote that
one.  Then $d$ vanishes at the odd one, and
\begin{align*}
\varepsilon_k\varepsilon_l
=(-1)^{\frac{k+l}{2}-1}
=-(-1)^{v}.
\end{align*}
Hence the two cases reduce to a single uniform formula
(if $w=u$, then $d(v)=0$ and $-(-1)^v=1$; if $w=v$, then $d(u)=0$ and
$-(-1)^v=-1$):
\begin{align*}
\varepsilon_k\varepsilon_lB_{kl}=\frac{1}{2} d(w),
\quad\text{hence}\quad
C_{kl}=\frac{1}{18} d(w).
\end{align*}
The closed form of $d$ now yields the entire table.
For $k=l$ we have $w=u=0$, so
\begin{align*}
C_{kk}=\frac{3}{18}=\frac{1}{6}.
\end{align*}
For $k<l$, the ranges $-5\le u\le-1$ and $2\le v\le10$ show that
$d(w)\ne0$ forces $w\equiv\pm2\pmod{12}$, that is, $u=-2$ or
$v\in\{2,10\}$; equivalently
\begin{align*}
l-k=4
\quad\text{or}\quad
k+l\in\{4,20\}.
\end{align*}
These conditions give precisely the six pairs in 
\begin{align*}
\mathcal S
=\{(1,3),(1,5),(3,7),(5,9),(7,11),(9,11)\},
\end{align*}
each with
\begin{align*}
C_{kl}=\frac{1}{18}\cdot\Bigl(-\frac{3}{2}\Bigr)=-\frac{1}{12},
\end{align*}
and every other coefficient with both indices odd vanishes. 
This proves the table and covers all three cases based on parity.
\end{proof}

\begin{proof}[Proof of Theorem~\ref{thm:p11}]
By the spectral decomposition and Lemma~\ref{lem:coeff},
\begin{align*}
\langle Y,M(t)\rangle_F
=\sum_{k=1}^{11}C_{kk}
+2\sum_{1\le k<l\le11}
C_{kl}\cos\bigl((\theta_k-\theta_l)t\bigr).
\end{align*}
The constant term equals $6\cdot\frac{1}{6}=1$.  Also,
\begin{align*}
\theta_{12-k}=-\theta_k,
\quad
\theta_6=0,
\quad
\theta_1=\theta_3+\theta_5.
\end{align*}
The last identity is precisely
\begin{align*}
\cos15^{\circ}=\cos45^{\circ}+\cos75^{\circ}.
\end{align*}
The nonzero off-diagonal coefficients in Lemma~\ref{lem:coeff} group as follows:
\begin{align*}
\begin{array}{c|c|c}
\text{common difference}
         & C_{kl}=1/12 & C_{kl}=-1/12\\ \hline
\theta_5 & (5,6),(6,7) & (1,3),(9,11)\\
\theta_3 &(3,6),(6,9) & (1,5),(7,11)\\
\theta_1 & (1,6),(6,11) & (3,7),(5,9)
\end{array}
\end{align*}
For each of the three differences, the sum of the corresponding coefficients
is zero.  All nonconstant terms therefore cancel, and
\begin{align*}
\langle Y,M(t)\rangle_F=1
\quad\text{for every }t\in\mathbb R.
\end{align*}
Meanwhile, 
\begin{align*}
\sum_{1 \le p,q \le 11}Y_{pq}=12,
\end{align*}
so
\begin{align*}
\left\langle Y,\frac{1}{11}\J\right\rangle_F=\frac{12}{11}\ne1.
\end{align*}
Lemma~\ref{lem:linf} now proves the theorem.
\end{proof}

\begin{remark}
We leave open whether there are infinitely many paths $P_n$ that
admit no probability measure yielding uniform average mixing.
\end{remark}

\end{document}